\pdfoutput=1
\documentclass[notitlepage,11pt]{article}%
\usepackage{amsmath}
\usepackage{amsfonts}
\usepackage{amssymb}
\usepackage{graphicx}
\usepackage{multicol}
\usepackage{breakcites}%
\providecommand{\U}[1]{\protect\rule{.1in}{.1in}}
\newtheorem{theorem}{Theorem}

\newtheorem{definition}[theorem]{Definition}

\newtheorem{problem}[theorem]{Problem}
\newtheorem{proposition}[theorem]{Proposition}
\newtheorem{remark}[theorem]{Remark}

\newtheorem{assumption}[theorem]{Assumption}
\newenvironment{proof}[1][Proof]{\noindent\textbf{#1.} }{\ \rule{0.5em}{0.5em}}

\begin{document}

\title{Solution of parabolic free boundary problems using transmuted heat polynomials}
\author{Igor V. Kravchenko\\{\footnotesize {Instituto Universit\'{a}rio de Lisboa (ISCTE-IUL),
Edif\'{\i}cio II, Av. Prof. An\'{\i}bal Bettencourt,}}\\{\footnotesize {1600-189 Lisboa, Portugal,} \texttt{ivkoh@iscte.pt}, }
\and Vladislav V. Kravchenko\\{\footnotesize {Departamento de Matem\'{a}ticas, CINVESTAV del IPN, Unidad
Quer\'{e}taro, Libramiento Norponiente No. 2000,}}\\{\footnotesize {Fracc. Real de Juriquilla,
Quer\'{e}taro, Qro. C.P. 76230 M\'{e}xico,} \texttt{vkravchenko@math.cinvestav.edu.mx}, }
\and Sergii M. Torba\\{\footnotesize {Departamento de Matem\'{a}ticas, CINVESTAV del IPN, Unidad
Quer\'{e}taro, Libramiento Norponiente No. 2000,}}\\{\footnotesize {Fracc. Real de Juriquilla,
Quer\'{e}taro, Qro. C.P. 76230 M\'{e}xico}, \texttt{storba@math.cinvestav.edu.mx}}}
\maketitle

\begin{abstract}
A numerical method for free boundary problems for the equation
\begin{equation}
u_{xx}-q(  x)  u=u_{t} \label{Abs1}%
\end{equation}
is proposed. The method is based on recent results from transmutation
operators theory allowing one to construct efficiently a complete system of
solutions for equation (\ref{Abs1}) generalizing the system of heat
polynomials. The corresponding implementation algorithm is presented.

\end{abstract}

\section{Introduction}

Free boundary problems (FBPs) for parabolic equations are of considerable
interest in physics (e.g., the Stefan problem) and in financial mathematics
(e.g., the problem of pricing of an American option). One of the relatively
simple and practical methods proposed for solving FBPs involving the heat
equation is the heat polynomials method (see \cite{colton1976solution},
\cite{Reemtsen1982}, \cite{colton1984numerical}, \cite{ReemtsenKirsch1984},
\cite{sarsengeldin2014analytical}, \cite{kharin2016analytical}) based on the
fact that the system of the heat polynomials represents a complete family of
solutions of the heat equation. In the book \cite{colton1976solution} D.
Colton proposed to extend this method onto parabolic equations with variable
coefficients by constructing an appropriate transmutation operator and
obtaining with its aid the corresponding transmuted heat polynomials. However,
the construction of the transmutation operator is a difficult problem itself.
In the present work we show that the transmuted heat polynomials required for
Colton's approach can be constructed without knowledge of the transmutation
operator by a simple and robust recursive integration procedure. To this aim a
recent result from \cite{CamposKravchenkoTorba_2011} concerning a mapping
property of the transmutation operators is used.

This makes possible to extend the heat polynomials method onto equations of
the form
\begin{equation}
u_{xx}-q(x)u=u_{t}. \label{Intro Eq}%
\end{equation}
We mention that linear parabolic equations of a more general form with coefficients depending on one variable reduce to \eqref{Intro Eq} (see, e.g., \cite[Chap.\ 2]{colton1976solution} or \cite{miyazawa1989theory}).

Thus, we propose a numerical method for approximate solution of a class of
FBPs involving (\ref{Intro Eq}). This method of transmuted heat polynomials
will be designated by THP. The main aim of this paper is to explain it in
detail and to propose a simple to implement algorithm for its application.

The subject of FBPs appears in many different fields and applications, as
such, presents a large variety of formulations and still open questions. We
will not be focusing on the existence and uniqueness of the solution in this
paper assuming that it exists in the classical sense.

The paper is structured as follows. In Section \ref{Sec_FP} we state the FBP.
In Section \ref{Sec_Theory} we define transmutation operators and present some
of their properties motivating the development of the THP method. In Section
\ref{Sec_Al} we construct the conceptual algorithm for the implementation of
the THP method for solving FBPs. In Section \ref{Sec_NumEx} we consider an
example with an exact solution, with its aid we illustrate the performance of
the method. In Section \ref{Sec_DiffFBP} we discuss the possibilities of
generalization of our construction and its application to more general FBPs.

\section{Statement of the problem\label{Sec_FP}}

Consider the differential expression
\[
\mathbf{A}=\frac{\partial^{2}}{\partial x^{2}}-q(x)
\]
with $q\in C[0,L]  $, $q:[  0,L]  \rightarrow
\mathbb{C}$.

\begin{figure}[ptb]
\centering
\includegraphics[bb = 0 0 190 89,
height=4.58cm,
width=9.73cm
]{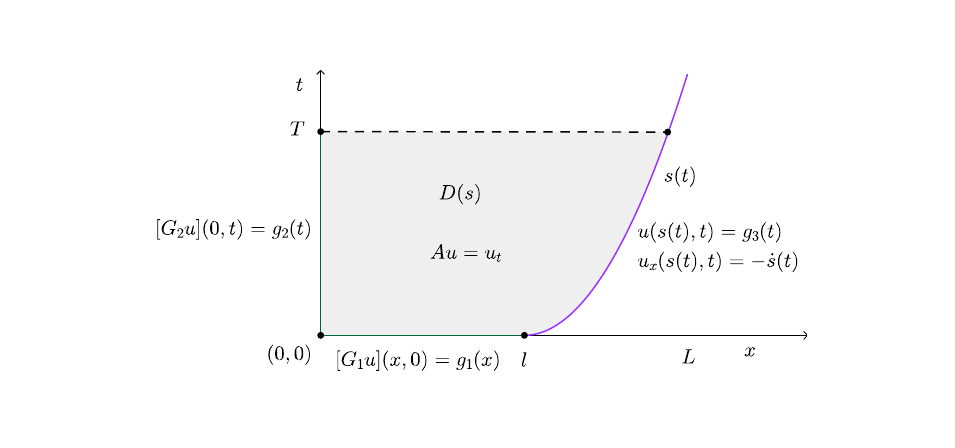}
\caption{Illustration of Problem \ref{Problem_FBP} statement}
\label{fig_St_Prob}
\end{figure}

Every $s\in C^{1}[  0,T]  $, such that $0<s(x)\le L$ for all $x\in [0,T)$ and $s(0)  =l$,
defines a domain%
\begin{equation}
D(  s)  =\left\{  (x,t)\in\mathbb{R}^{2}:0<x<s(t),0<t<T\right\}  ,
\label{D(s)}%
\end{equation}
see Figure \ref{fig_St_Prob}. Consider the first order linear differential
operators $\mathbf{G}_{1}=\gamma_{11}(  x)  +\gamma_{12}(
x)  \frac{\partial}{\partial x}$ and $\mathbf{G}_{2}=\gamma_{21}(
t)  +\gamma_{22}(  t)  \frac{\partial}{\partial x}$ , where
$\gamma_{ij}$ are some given continuous functions.

\begin{problem}
\label{Problem_FBP}Find functions $(  u,s)  $ such that
\begin{description}
\item[(i)] $s\in C^{1}[  0,T]$ and such that $0<s(x)\le L$ for all $x\in [0,T)$ and $s(0)  =l$,
\item[(ii)] $u\in C^{2,1}\bigl(  D(  s)  \bigr)$,
\item[(iii)] the following equation is satisfied on $D(  s)  $
\begin{equation}
\mathbf{A}u(  x,t)  =u_{t}(  x,t)  ,\qquad(  x,t)  \in D(  s)  , \label{eq_FBP1}%
\end{equation}
\item[(iv)] and the following boundary conditions are satisfied
\begin{align}
\left[\mathbf{G}_{1}u\right]  (  x,0)   &  =g_{1}(
x)  ,\qquad x\in(  0,l),\qquad l\leq
L,\label{eq_FBP_b1}\\
\left[  \mathbf{G}_{2}u\right]  (  0,t)   &  =g_{2}(
t)  ,\qquad t\in(  0,T)  ,\label{eq_FBP_b2}\\
u\bigl(  s(  t)  ,t\bigr)   &  =g_{3}(  t)  ,\qquad t\in(  0,T)  ,\label{eq_FBP_b3}\\
u_{x}\bigl(  s(  t)  ,t\bigr)   &  =-\dot{s}(  t)
,\qquad t\in(  0,T),   \label{eq_FBP_b4}%
\end{align}
here the dot over the function $s$ means the derivative with respect to the variable $t$. The last condition is usually known as the equation of heat balance or as the Stefan condition.
\end{description}
\end{problem}

This problem is broadly studied in literature (see, e.g., \cite{Friedman1964}, \cite{rubinshteuin1971stefan}, \cite{crank1984free},
\cite{meirmanov1992stefan}, \cite{FASANO1979247} and
\cite{tarzia2000bibliography} for additional bibliography). In particular, the
classical one dimensional one phase Stefan problem is a special case of
Problem \ref{Problem_FBP}. Since the subject of this paper is the approximate numerical method for solution of Problem \ref{Problem_FBP}, we make the following assumption.

\begin{assumption}
There exists a unique solution to Problem \ref{Problem_FBP}.
\end{assumption}

A relevant example of an existence and uniqueness result is given in \cite[Theorem 1]{FASANO1979247}, see Remark \ref{Remark Existance} below.

\section{Transmuted heat polynomials\label{Sec_Theory}}

\subsection{Heat polynomials}

The heat polynomials are defined for $n\in\mathbb{N}_0$ as (see, e.g.,
\cite{rosenbloom1959expansions} and \cite{widder1962})
\[
h_{n}(  x,t)  =\sum_{k=0}^{\left[  n/2\right]  }
c_{k}^{n}x^{n-2k}t^{k},\qquad \text{where }c_{k}^{n}=\frac{n!}{(n-2k)!k!}
\]
and $\left[  \cdot\right]  $ denotes the entire part of the number.
The first five heat polynomials are%
\begin{align*}
h_{0}(  x,t)   &  =1,\qquad h_{1}(  x,t)
=x,\qquad h_{2}(  x,t)  =x^{2}+2t,\\
h_{3}(  x,t)   &  =x^{3}+6xt,\qquad h_{4}(  x,t)
=x^{4}+12x^{2}t+12t^{2}.
\end{align*}

The set of the heat polynomials $\left\{  h_{n}\right\}  _{n\in\mathbb{N}_0}$ represents a complete system of solutions for the
heat equation%
\begin{equation*}
u_{xx}=u_{t} %\label{Eq_Heat}%
\end{equation*}
on any domain $D(s)$ defined by (\ref{D(s)}) (see
\cite{COLTON1977}).

\subsection{Formal powers and the transmutation operator}

Let $f$ be a nonvanishing (in general, complex valued) solution of the equation
\begin{equation}
\left(  \frac{d^{2}}{dx^{2}}-q(  x)  \right)  f(x)=0,\qquad
x\in(0,L), \label{Eq_tr1}%
\end{equation}
such that
\begin{equation}
f(  0)  =1. \label{Eq_tr2}%
\end{equation}
The existence of such solution\footnote{In fact the only reason for the
requirement of the absence of zeros of the function $f$ is to make sure that
the auxiliary functions (\ref{eq_phik}) be well defined. As was shown in
\cite{KravchenkoTorba2014Mod} this can be done even without such requirement,
but corresponding formulas are somewhat more complicated.} for any
complex valued $q\in C[0,L]$ was proved in \cite{KravchenkoPorter10} (see also
\cite{camporesi2011generalization}).

Consider two sequences of recursive integrals (see \cite{Kravchenko09},
\cite{KravchenkoMorelosTorba14}, \cite{KravchenkoPorter10})
\[
X^{(0)}(x)\equiv1,\qquad X^{(n)}(x)=n\int_{0}^{x}X^{(n-1)}(s)\left(
f^{2}(s)\right)  ^{(-1)^{n}}\,\mathrm{d}s,\qquad n=1,2,\ldots
\]
and
\[
\widetilde{X}^{(0)}\equiv1,\qquad\widetilde{X}^{(n)}(x)=n\int_{0}%
^{x}\widetilde{X}^{(n-1)}(s)\left(  f^{2}(s)\right)  ^{(-1)^{n-1}}%
\,\mathrm{d}s,\qquad n=1,2,\ldots.
\]
\begin{definition}
\label{Def_FormalPowers}The family of functions $\left\{  \varphi_{n}\right\}
_{n=0}^{\infty}$ constructed according to the rule%
\begin{equation}
\varphi_{n}(  x)  =
\begin{cases}
f(  x)  X^{(  n)  }(  x)  , &   n\text{
odd},\\
f(  x)  \tilde{X}^{(  n)  }(  x)  , &
n\text{ even},
\end{cases}
  \label{eq_phik}%
\end{equation}
is called the system of \textbf{formal powers} associated with $f$.
\end{definition}

The formal powers arise in the spectral parameter power series
(SPPS)\ representation for solutions of the Sturm-Liouville equation, see
\cite{kravchenko2008representation}, \cite{KravchenkoPorter10},
\cite{khmelnytskaya2015eigenvalue}, \cite{KravchenkoMorelosTorba14}.

The following result from \cite{marchenko1952some} (see also
\cite{kravchenko2015representation} and \cite{kravchenko2015analytic} for
additional details) and from \cite{CamposKravchenkoTorba_2011} guarantees the existence of a transmutation operator associated with $f$ and shows its connection with the system of formal powers.

\begin{theorem}
\label{Teor_TrOp}Let $q\in C[  0,L]  $. Then there exists a unique
complex valued function $K(  x,y)  \in C^{1}\bigl(  [
0,L]  \times[  -L,L]  \bigr)  $ such that the Volterra
integral operator%
\[
\mathbf{T}u(  x)  =u(  x)  +\int_{-x}^{x}K(
x,y)  u(  y) \, dy
\]
defined on $C[  0,L]  $ satisfies the equality
\[
\mathbf{AT}[  u]  =\mathbf{T}\left[  \frac{d^{2}u}{dx^{2}}\right]
\]
for any $u\in C^{2}[  0,L]  $ and
\[
\mathbf{T}\left[  1\right]  =f.
\]
\end{theorem}

Note that if $v=\mathbf{T}u$ then $v(0)=u(0)$ and $v^{\prime}(0)=u^{\prime
}(0)+f^{\prime}(0)u(0)$.

\begin{theorem}[\cite{CamposKravchenkoTorba_2011}]\label{ThFormPowers}
\[
\mathbf{T}\left[  x^{n}\right]  =\varphi_{n}(  x)  \qquad \text{for
any }n\in\mathbb{N}_0.
\]
\end{theorem}

\subsection{Construction of the transmuted heat polynomials}

Denote $H_{n}=\mathbf{T}\left[  h_{n}\right]  $. Thus, $H_{n}(x,t)=h_{n}%
(x,t)+\int_{-x}^{x}K(  x,y)  h_{n}(  y,t)\,  dy$. The
functions $H_{n}$ are called the \textbf{transmuted heat polynomials}  \cite{KJT2017}.

We have $\left(  \mathbf{A}-\frac{\partial}{\partial t}\right)  H_{n}%
=\mathbf{AT}h_{n}-\frac{\partial}{\partial t}\mathbf{T}h_{n}=\mathbf{T}\left(
\frac{\partial^{2}}{\partial x^{2}}-\frac{\partial}{\partial t}\right)
h_{n}=0$ and hence every $H_{n}$ is a solution of the equation
\begin{equation}
\mathbf{A}u=u_{t}. \label{eq_heatEqOnly}%
\end{equation}

The set $\left\{  H_{n}\right\}  _{n\in\mathbb{N}_0  }$ is a complete system of solutions of (\ref{eq_heatEqOnly})
 in the following sense.
\begin{proposition}\label{Prop Complete system}
Let $u$ be a classical solution of \eqref{eq_heatEqOnly} in $D(s)$, continuous in the closure $\bar D(s)$. Then for any compact set $D_c\subset D(s)$ and any given $\varepsilon>0$ there exist $N$ and constants $a_0,\ldots,a_N$ such that
\begin{equation*}
\max_{(x,t)\in D_c}\left| u(x,t) - \sum_{n=0}^N a_nH_n(x,t)\right|<\varepsilon.
\end{equation*}
Moreover, if $s(t)$ can be extended to a function analytic in the disk $\{t\in \mathbb{C}:|t|\le T\}$, the uniform approximation property is valid on the whole set $\bar D$.
\end{proposition}
\begin{proof}
Suppose first that $q\in C^1[0,L]$. Note that any compact set $D_c$ can be covered by a subset $D_a\subset \bar D(s)$ of the form $D_a=\{(x,t)\in\mathbb{R}^2: 0\le x\le x_1(t)\le s(t),\ 0\le t\le T\}$
with analytic function $x_1(t)$. For the set $D_a$ the proof is completely similar to that of \cite[Theorem 2.3.3]{colton1976solution} with the only change that the transmutation operator $\mathbf{T}$ and its inverse are used. For the general case $q\in C[0,L]$ one approximates $q$ by a $C^1$ function.
\end{proof}

\begin{theorem}
\label{Cor_Hn} The transmuted heat polynomials admit the following form%
\[
H_{n}(  x,t)  =\sum_{k=0}^{\left[  n/2\right]  }
c_{k}^{n}\varphi_{n-2k}(  x)  t^{k}.
\]
\end{theorem}

\begin{proof}
This equality is an immediate corollary of Theorem \ref{ThFormPowers}. Indeed,
we have $H_{n}(  x,t)  =\mathbf{T}h_{n}(  x,t)  =\sum_{k=0}^{\left[  n/2\right]  }c_{k}^{n}\mathbf{T}\left[  x^{n-2k}\right]  t^{k}$, where Theorem
\ref{ThFormPowers} is used.
\end{proof}

The explicit form of the functions $H_{n}$ presented in this theorem makes
possible the construction of the approximate solution to Problem
\ref{Problem_FBP} by the THP method.

\section{Description of the method\label{Sec_Al}}

We proceed to the step by step construction of the THP method, summarizing the
algorithm at the end of the section.

Assume that we have already calculated the formal powers $\varphi_{k}$ and the
functions $H_{n}$. Further, let $N$ be the highest index of the formal power
considered, or equivalently the highest degree in $x$ of the considered heat
polynomials. We denote by
\begin{equation}
u_{N}(  x,t)  = \sum_{n=0}^{N}
a_{n}H_{n}(  x,t)  \label{Eq_SolAprox}%
\end{equation}
the approximation of the solution $u(  x,t)  $, and by $\bar
{a}=\left(  a_{0},...,a_{N}\right)^T$ the column-vector of the unknown coefficients.
Note that by construction functions $u_{N}$ satisfy equation \eqref{eq_FBP1},
all we need is to find suitable coefficients $\bar{a}$.

Denote by $\left\{  t_{i}\right\}  $ an ordered set of $N_{t}+1$ points of the
interval $[  0,T]  $, with $t_{0}=0<t_{1}<...<t_{N_{t}}=T$,
$\bar{t}=\left(  t_{0},...,t_{N_{t}}\right)^T  $, and by $\left\{
x_{i}\right\}  $ an ordered set of $N_{x}+1$ points of the initial boundary,
with $x_{0}=0<x_{1}<...<x_{N_{x}}=l$, $\bar{x}=\left(  x_{0},...,x_{N_{x}%
}\right)^T  $.

Consider a set of $K+1$ linearly independent differentiable functions,
$\beta_{k}:[  0,T]  \rightarrow\mathbb{R}$. We are looking for
the free boundary in the form
\[
s_{K}(  t)  =\sum_{k=0}^{K}
b_{k}\beta_{k}(  t)  .
\]
Denote the vector of the unknown coefficients\footnote{It is possible to search for the free boundary in a more general form, see Section \ref{SubSec_NL_FreeBoundary}.} by $\bar{b}=\left(
b_{0},...,b_{K}\right)^T$.
For any function $v$ defined on the set of points $\left\{  y_{i}\right\}
_{i=0,1,...,N_{y}}$ the following notation is introduced
\[
v\left(  \bar{y}\right)  =\bigl(  v\left(  y_{0}\right)  ,...,v\left(
y_{N_{y}}\right)  \bigr)^T.
\]

A numerical approximation problem consists in finding a set of the
coefficients $\left(  \bar{a},\bar{b}\right)  $ that best fits the conditions
of Problem \ref{Problem_FBP} with the following norm chosen%\footnote{All functions that we are considering here are real valued.}
\[
\left\Vert v(  \bar{y})  \right\Vert ^{2}=
\sum_{i=0}^{N_{y}}
|v(  y_{i})|^{2}.
\]
The following magnitudes are to be minimized,%
\begin{align*}
I_{1}\left(  \bar{a},\bar{b}\right)   &  =\left\Vert \left[  \mathbf{G}%
_{1}u_{N}\right]  \left(  \bar{x},0\right)  -g_{1}(  \bar{x})
\right\Vert ,\\
I_{2}\left(  \bar{a},\bar{b}\right)   &  =\left\Vert \left[  \mathbf{G}%
_{2}u_{N}\right]  \left(  0,\bar{t}\right)  -g_{2}(  \bar{t})
\right\Vert ,\\
I_{3}\left(  \bar{a},\bar{b}\right)   &  =\left\Vert u_{N}\bigl(  s_{K}\left(
\bar{t}\right)  ,\bar{t}\bigr)  -g_{3}(  \bar{t})  \right\Vert ,\\
I_{4}\left(  \bar{a},\bar{b}\right)   &  =\left\Vert \left(  u_{N}\right)
_{x}\bigl(  s_{K}\left(  \bar{t}\right)  ,\bar{t}\bigr)  +\dot{s}_{K}(\bar
{t})\right\Vert .
\end{align*}
Each of them is related to a boundary condition from
\eqref{eq_FBP_b1}--\eqref{eq_FBP_b4}. With introduction of the value function
\begin{equation}
F\left(  \bar{a},\bar{b}\right)  =\sum_{i=1}^{4}
I^2_{i}\left(  \bar{a},\bar{b}\right)  , \label{eq_F}%
\end{equation}
the minimization problem can be stated as follows.

\begin{problem}
\label{Problem_M}Find\footnote{For a function $f:X\rightarrow Y$, the
$\arg\min$ over a subset $S$ of $X$ is defined as%
\[
\underset{x\in S\subseteq X}{\arg\min}\, f(  x)  :=\left\{  x:x\in
S\wedge\forall y\in S:f(  y)  \geq f(  x)  \right\}  .
\]
}
\[
\underset{\left(  \bar{a},\bar{b}\right)  }{\arg\min}\,F\left(  \bar{a},\bar
{b}\right)  ,
\]
subject to
\begin{equation}
0<s_{K}( t)  \le L,\qquad t\in [0,T]. \label{eq:_probM_condS}%
\end{equation}
\end{problem}

Note that instead of the uniform norm chosen in \cite{Reemtsen1982} for the heat polynomials method, we used the $L_2$ norm in the value function \eqref{eq_F}. The main reason for such choice was to take advantage of the presence of the so-called separable linear parameters
(see \cite[Chap.\ 6.2]{ross1990nonlinear}, see also \cite{herreraporter2017}) and to reduce the number of parameters in the value function. Indeed, for each fixed $\bar b$, the constrained minimization Problem \ref{Problem_M} reduces to an unconstrained linear least squares problem for parameters $\bar a$ and can be easily solved exactly (see Subsection \ref{Subsect LinMin}). That is, for each $\bar b$ we can define
\begin{equation}\label{Linear LS}
\bar a(\bar b) :=\underset{\bar{a}}{\arg\min}\,F\left(  \bar{a},\bar
{b}\right).
\end{equation}
So instead of minimizing the value function $F$ over an $N+K+2$ dimensional space of parameters $(\bar a, \bar b)$, the problem can be reduced to minimization of the function
\[
\tilde F(\bar b):= F\bigl(\bar a(\bar b), \bar b\bigr)
\]
over a $K+1$ dimensional space. Such reformulation of the original optimization problem leads to a more robust convergence of numerical optimization algorithms, c.f.,  \cite{herreraporter2017}.  We state this new minimization problem as follows.

\begin{problem}
\label{Problem_Mtil}Find
\[
\underset{\bar{b}}{\arg\min}\,\tilde{F}\left(  \bar{b}\right)
\]
subject to
\begin{equation}
0<s_{K}( t)  \le L,\qquad t\in [0,T]. \label{eq:_probMtil_condS}
\end{equation}
\end{problem}

\subsection{Linear minimization problem}\label{Subsect LinMin}
Let a vector $\tilde{b}$ be fixed. We have to solve the linear least squares problem \eqref{Linear LS}. Let us denote the corresponding approximate boundary by $\tilde{s}$,
\[
\tilde{s}(  t)  =\sum_{i=0}^{N_{\beta}}
\tilde{b}_{i}\beta_{i}(  t)  .
\]
Note that problem \eqref{Linear LS} is equivalent to solving the overdetermined system of linear equations (resulting from the boundary conditions \eqref{eq_FBP_b1}--\eqref{eq_FBP_b4})
\begin{align*}
\left[  \mathbf{G}_{1}u_{N}\right]  \left(  \bar{x},0\right)   &  =\sum_{n=0}^{N}
a_{n}B_{n}\left(  \bar{x}\right)  , &
u_{N}\bigl(  \tilde{s}\left(  \bar{t}\right)  ,\bar{t}\bigr)   &  =\sum_{n=0}^{N}
a_{n}D_{n}\left(  \bar{t}\right)  ,\\
\left[  \mathbf{G}_{2}u_{N}\right]  \left(  0,\bar{t}\right)   &  =
\sum_{n=0}^{N}
a_{n}C_{n}\left(  \bar{t}\right)  ,&
\left(  u_{N}\right)  _{x}\bigl(  \tilde{s}\left(  \bar{t}\right)  ,\bar
{t}\bigr)   &  =\sum\limits_{n=0}^{N}
a_{n}E_{n}\left(  \bar{t}\right)  ,
\end{align*}
where
\begin{align}
B_{n}\left(  \bar{x}\right)   &  =
\begin{cases}
c_{0}^{n}\gamma_{11}(  \bar{x})  \varphi_{n}(  \bar{x})
+c_{\frac{n-1}{2}}^{n}\gamma_{12}(  \bar{x})  \varphi_{1}^{\prime
}(  \bar{x})  , &   n\text{ odd},\\
c_{0}^{n}\gamma_{11}(  \bar{x})  \varphi_{n}(  \bar{x})
,   & n\text{ even},
\end{cases}
\label{eq_CoefB}\\
C_{n}\left(  \bar{t}\right)   & =
\begin{cases}
c_{\frac{n-1}{2}}^{n}\gamma_{22}(  \bar{t})  \bar{t}^{\frac{n-1}%
{2}},   & n\text{ odd},\\
c_{\frac{n}{2}}^{n}\gamma_{21}(  \bar{t})  \bar{t}^{\frac{n}{2}}, &
 n\text{ even},
\end{cases} \label{eq_CoefC} \\
D_{n}(  \bar{t})   &  =\sum_{k=0}^{[  n/2]  }
c_{k}^{n}\varphi_{n-2k}(  \tilde{s}(  \bar{t})
\bar{t}^{k},\qquad
E_{n}(  \bar{t})     =\sum_{k=0}^{[  n/2]  }
c_{k}^{n}\varphi_{n-2k}^{\prime}(  \tilde{s}(  \bar{t})
)  \bar{t}^{k},
\label{eq_CoeffsDE}
\end{align}
or in the matrix form
\begin{equation}\label{eq ls}
\mathbf{B} \bar a \simeq \mathbf{g},
\end{equation}
where
\[
\mathbf{B} = \begin{pmatrix}
B_{0}(  \bar{x})  & \dots & B_{N}\left(  \bar{x}\right) \\
C_{0}(  \bar{t})  & \dots & C_{N}\left(  \bar{t}\right) \\
D_{0}(  \bar{t})  & \dots & D_{N}\left(  \bar{t}\right) \\
E_{0}(  \bar{t})  & \dots & E_{N}\left(  \bar{t}\right)
             \end{pmatrix}
             \qquad \text{and}\qquad
\mathbf{g} = \begin{pmatrix}
               g_{1}(  \bar{x}) \\
               g_{2}(  \bar{t}) \\
               g_{3}(  \bar{t}) \\
               \dot{\widetilde s} (  \bar{t})
             \end{pmatrix}.
\]
Note that the derivatives $\varphi'$ in \eqref{eq_CoeffsDE} do not require numerical differentiation and can be obtained in a closed form from \eqref{eq_phik}.

The solution of this overdetermined system coincides
\cite[Thm.\ 5.14]{MadsenNielsen2010} with the unique solution of the following fully determined system of linear equations
\begin{equation}\label{eq Normal LS}
\mathbf{C}\bar a = \mathbf{h},\qquad \text{where } \mathbf{C} = \mathbf{B}^T\mathbf{B}\text{ and }\mathbf{h} = \mathbf{B}^T \mathbf{g}.
\end{equation}

See also \cite{lawson1995solving} or \cite{wright2006numerical} for various methods of efficient solution of equation \eqref{eq ls} or \eqref{eq Normal LS}. In particular, we used the Moore-Penrose pseudo-inverse method.

\subsection{Implementation}

Here we present the algorithm of the implementation of the THP method
for Problem \ref{Problem_FBP}.

\begin{enumerate}
\item Find a particular solution of equation \eqref{Eq_tr1} satisfying
\eqref{Eq_tr2}. For example, the SPPS method presented in
\cite{KravchenkoPorter10} can be used.

\item Construct the formal powers on an interval from $0$ to some $L>l$. In this paper we represented all the functions involved by their values on a uniform mesh and used the Newton-Cotes six point integration rule. Since we may need the values of the formal powers at arbitrary points $\tilde{s}\left(  \bar{t}\right)\subset [0,L]$, we approximated the formal powers by splines passing through their values on the selected mesh. Spline integration can be used as well for the construction of the formal powers, c.f., \cite{khmelnytskaya2013wave}. See also \cite{kravchenko2015representation} for the discussion of other possible methods.

\item Compute the coefficients $B_{n}$ and $C_{n}$ given by \eqref{eq_CoefB}
and \eqref{eq_CoefC} and the respective functions $g_{1}$ and $g_{2}$. Since these
conditions are independent of the free boundary (i.e.\ of the choice of
$\tilde{b}$), they need to be computed only once.

\item Construct a function that solves problem \eqref{Linear LS} for each
given set of coefficients $\tilde{b}$.

\item Construct the value function $\tilde{F}\left(  b\right)  $ using
\eqref{eq_F}.

\item Solve the constrained minimization Problem \ref{Problem_Mtil} using any suitable algorithm, see, e.g., \cite{wright2006numerical}. For the numerical illustration we used the Matlab function \texttt{fmincon}. Note that for faster convergence we may initially solve the minimization Problem \ref{Problem_Mtil} for some small value $K'<K$ and use the coefficients $b^0_k$, $k=0,\ldots,K'$ obtained as an initial value for the optimization algorithm for larger value $K$.
\end{enumerate}

\section{Numerical illustration\label{Sec_NumEx}}

\subsection{Example with an exact solution and a tractable free boundary}

For the numerical experiment we constructed an example of an FBP admitting an exact solution.

For every $s\in C^{1}[  0,1]  $, with $s(  0)  =1$
consider the domain
\[
D(s)=\left\{  (x,t)\in\mathbb{R}^{2}:0<x<s(t),\, 0<t<1\right\}  .
\]

\begin{problem}
\label{Problem_E}Find the pair $(  s,u)  $ such that $s\in
C^{1}[  0,1]$, $u\in C^{2,1}\left(  D(
s)  \right)  $ and
\[
u_{xx}-x^{2}u=u_{t},\qquad (x,t)\in D(s),
\]
with the following boundary conditions
\begin{align}
u(  x,0)   &  =e^{-\frac{x^{2}}{2}},& x\in[0,1]  , \label{Eq_ex615a}\\
u_{x}(0,t)  &  =0,& t\in(  0,1)  ,\label{Eq_ex615b}\\
u(s(t),t)  &  =e^{-\operatorname{Ei}^{-1}\left(  2C-2e^{-t} \right)  -t},& t\in(0,1)  ,\label{Eq_ex615c}\\
u_{x}(s(t),t)  &  =-\dot s (t),& t\in(0,1)  ,\label{Eq_ex615d}
\end{align}
where
\[
C=\frac{1}{2}\operatorname{Ei}\left(  \frac{1}{2}\right)  +1\approx1.2271,
\]
and $\operatorname{Ei}^{-1}$ stands for the inverse function of the exponential integral
$\operatorname{Ei}$,%
\[
\operatorname{Ei}\left(  x\right)  =-\operatorname{v.p.}\int_{-x}^{\infty}\frac{e^{-t}}{t}dt.
\]
\end{problem}

Problem \ref{Problem_E} possesses an exact solution $u$ given by
\begin{equation}
u(x,t)=e^{-\frac{x^{2}}{2}-t},\qquad(x,t)\in D(s), \label{Eq_ex2}%
\end{equation}
and the free boundary
\begin{equation}
s\left(  t\right)  =\sqrt{2\operatorname{Ei}^{-1}\left(  2
C-2e^{-t} \right)  },\qquad t\in\left[
0,1\right]  . \label{eq_ExactBoundary}%
\end{equation}

\begin{remark}\label{Remark Existance}
The existence and the uniqueness of the solution for  $t\in \left[
0,1\right]$ is guaranteed by \cite[Theorem 1]{FASANO1979247}. For the
application of the theorem it is
necessary to use the transformation $v(x,t)=u(x,t)+\gamma(  t)  $, where $\gamma(t)
=\exp\left(-\operatorname{Ei}^{-1}\left(  2C-2e^{-t}\right)-t\right)$.
\end{remark}

\subsection{Numerical illustration}

We proceed by presenting numerical results delivered by the algorithm
described in Section \ref{Sec_Al}. All the
calculations were carried out in Matlab R2012a.

A particular solution to equation \eqref{Eq_tr1} and the formal powers \eqref{eq_phik} were represented by their values on a 2000-points uniform mesh. The Newton-Cotes integration rule was used for their computation. The number of formal powers considered was set at $N=12$. Finally the function \texttt{spapi} was used to create splines approximating the formal powers. For computing the $\operatorname{Ei}^{-1}$ we have used a polynomial interpolation on a fine grid for the function $\operatorname{Ei}$. The values for the function $\operatorname{Ei}$ were constructed by the series expansions presented in \cite[(8.214)]{GradshteynRyzhik07}, see also \cite{pecina1986function}.

The sets of points $\left\{  t_{i}\right\}  $ and $\left\{  x_{i}\right\}  $
considered are both equidistant grids of $101$ points. The
free boundary $s$ is sought in terms of polynomials. The condition $s(0)  =1$ inspires the following form
\[
s_{K}\left(  \bar{t}\right)  =1+\sum_{j=1}^{K}
b_{j}\bar{t}^{j}.
\]
In our calculations we have considered $K=6$.

The linear problem \eqref{Linear LS} was solved by using the Matlab function
\texttt{pinv}. We have found the solution to Problem \ref{Problem_Mtil} using
the Matlab routine \texttt{fmincon}.

The proposed algorithm converged rapidly for various initial free boundaries tested. On Figure \ref{fig FreeBoundaries} we present
one of the initial boundaries, the exact free boundary $s(t)$ calculated from equation \eqref{eq_ExactBoundary} and the difference between the exact free boundary $s(t)$ and the obtained approximation $s_K(t)$.
The obtained approximate boundary was
\begin{equation}
s_{K}(t)  =1
+0.60657885t
-0.30458770t^2
+0.03631846t^3
+0.06761711t^4
-0.05111378t^5
+0.01270860t^6, \label{Eq_AproxSolution}%
\end{equation}
here and after the coefficients are presented up to the eighth decimal place.
The coefficients of the vector $\bar{a}$ that corresponds to the solution of
the linear problem \eqref{Linear LS} with $\bar{b}$ consisting of the
coefficients from \eqref{Eq_AproxSolution} are
\begin{align*}
a_0 &= 1.00000201, & a_1 &= -3.88882172\cdot 10^{-6}, & a_2 &=  -5.00020660\cdot 10^{-1},\\
a_3 &= 2.67326664\cdot 10^{-5}, & a_4 &= 4.16822991\cdot 10^{-2}, & a_5 &= -2.03778980\cdot 10^{-5}, \\
a_6 & = -1.38990292\cdot 10^{-3}, & a_7 &= 4.02691726\cdot 10^{-6}, & a_8 &= 2.43785480\cdot 10^{-5}, \\
a_9 &=-2.56424452\cdot 10^{-7}, & a_{10} &= -2.28538929\cdot 10^{-7}, & a_{11} &= 4.70925279\cdot 10^{-9}, \\
a_{12} &= 7.12497038\cdot 10^{-10}.
\end{align*}

\begin{figure}[ptb]
\centering
\includegraphics[bb=198 309 414 482, height=2.4in, width=3in]{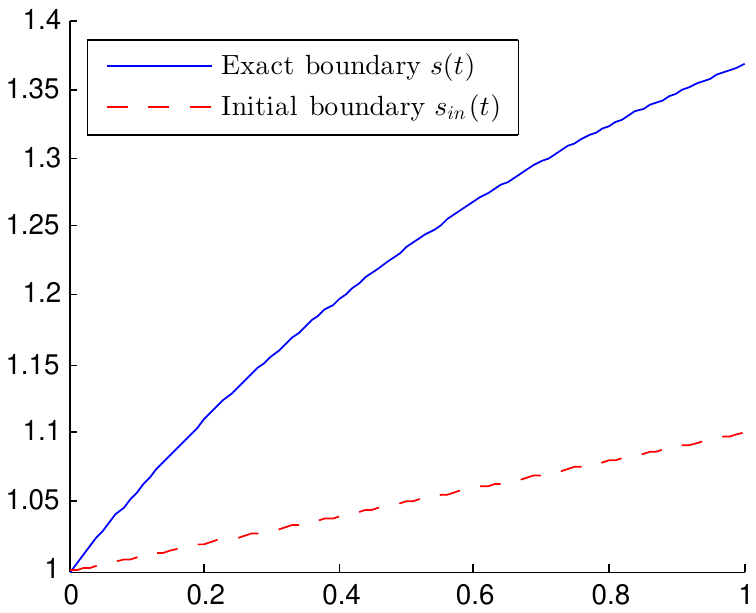}\quad
\includegraphics[bb=198 309 414 482, height=2.4in, width=3in]{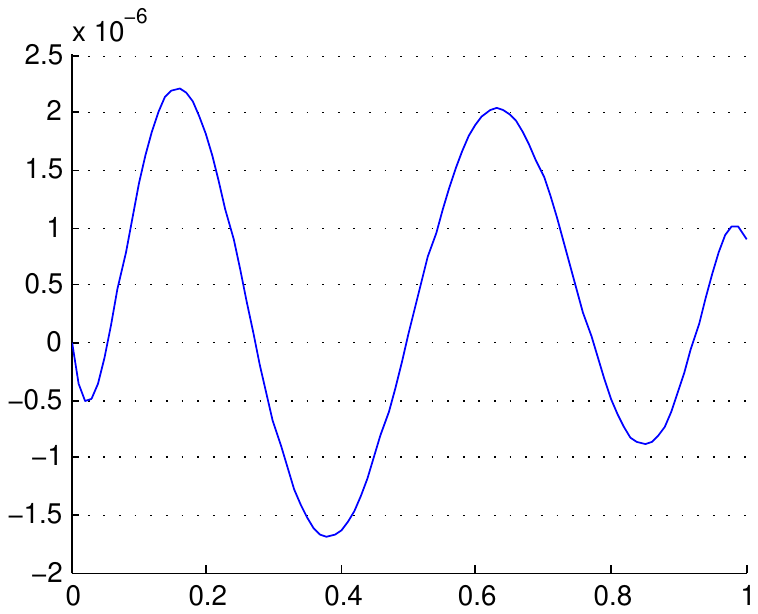}
\caption{Left: the initial boundary $s_{in}(t)=1+0.1t$ and the exact boundary $s(t)$ for the Problem \ref{Problem_E}. Right: the difference between the exact boundary $s(t)$ and the approximate boundary $s_K(t)$.
}
\label{fig FreeBoundaries}
\end{figure}

\begin{figure}[ptb]
\centering
\includegraphics[bb=198 309 414 482, height=2.4in, width=3in]{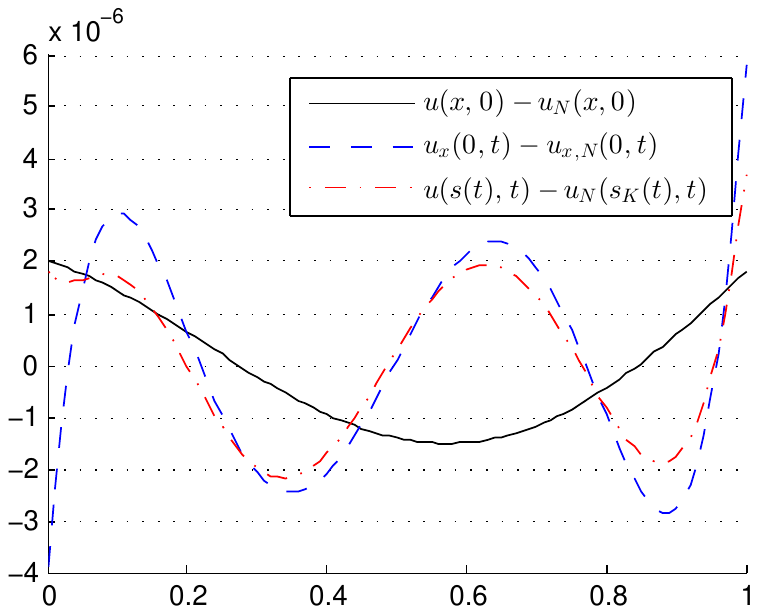}\quad
\includegraphics[bb=198 309 414 482, height=2.4in, width=3in]{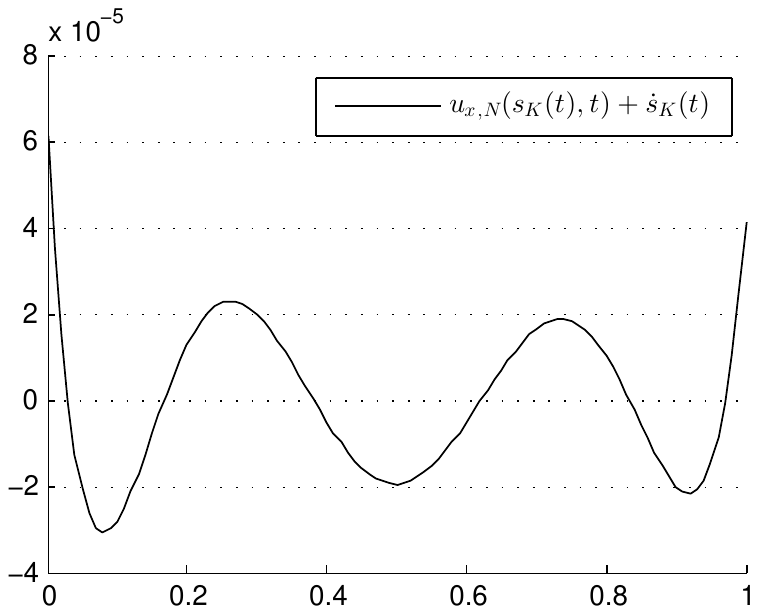}
\caption{
Accuracy of fulfillment of the boundary conditions \eqref{Eq_ex615a}--\eqref{Eq_ex615d} for the approximate solution
$u_N$ and approximate free boundary $s_K$ obtained by the algorithm. On the left: conditions \eqref{Eq_ex615a}--\eqref{Eq_ex615c}, on the right: condition \eqref{Eq_ex615d}.}
\label{Fig DiffBoundaries}
\end{figure}

Figure \ref{Fig DiffBoundaries} illustrates the accuracy of fulfillment of the boundary conditions \eqref{Eq_ex615a}--\eqref{Eq_ex615d}.
The absolute value of the difference between the exact solution
\eqref{Eq_ex2} and the obtained approximate solution in the domain $D(s)$ is
presented on Figure \ref{Fig max u u_N}.
\begin{figure}[ptbh]
\centering
\includegraphics[height=3.6in, width=5in]
{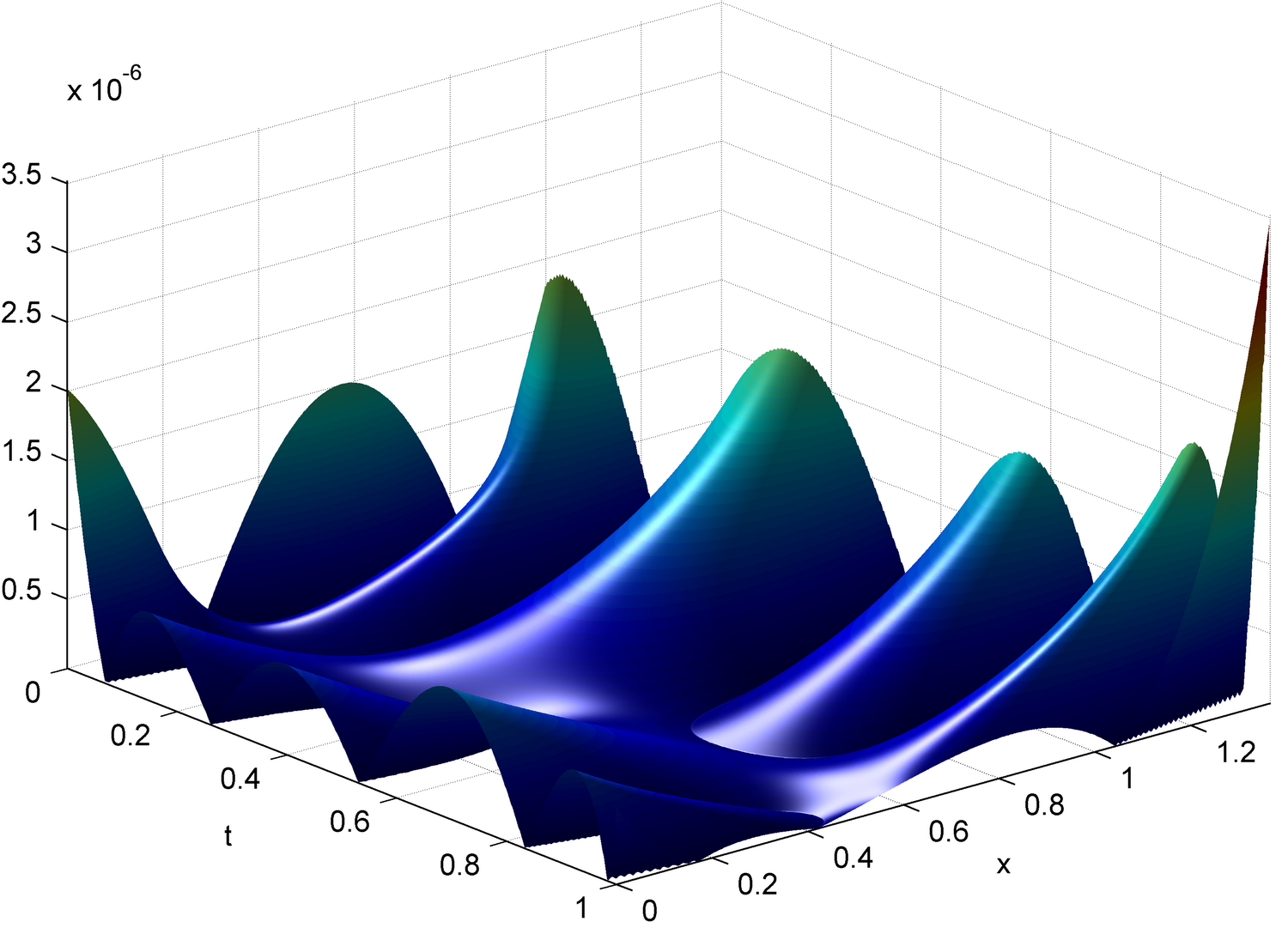}
\caption{The absolute difference $| u(x,t)-u_{N}(x,t)|$  between the exact and the approximate solutions for the Problem \ref{Problem_E} in the domain $D(s)$.}
\label{Fig max u u_N}
\end{figure}

Note that for the problem considered one may look for an approximate solution in the form
\[
u_{N}(  x,t)  = \sum_{n=0}^{N/2}
a_{2n}H_{2n}(  x,t),
\]
i.e., having only even coefficients $a_{2n}$. In such way the boundary condition \eqref{Eq_ex615b} is satisfied automatically. The proposed algorithm can be applied with minimal modifications for such simplified form of an approximate solution, however we did not observe any gain in the obtained result, the original formulation of the algorithm performed equally well.

\section{Possible extensions of the method to more general FBPs and final
remarks\label{Sec_DiffFBP}}

\subsection{Generalization of the operator $\mathbf{A}$}
The presented method can be extended onto operators of the form
\[
\mathbf{C}=\alpha_{1}(x)\frac{\partial^{2}}{\partial x^{2}%
}+\alpha_{2}(x)  \frac{\partial}{\partial x}+\alpha_{3}(x)
\]
with sufficiently regular coefficients $\alpha_{i}(x)$. For the
construction of the formal powers in this case see
\cite{KravchenkoTorba_Neumann_16}. The FBPs with this type of operators are
very common in financial applications.

\subsection{Generalizations of the conditions on a free boundary}

\subsubsection{Linear generalization}

The conditions \eqref{eq_FBP_b3} and \eqref{eq_FBP_b4} are sometimes referred
to as Stefan's conditions. Our algorithm can be applied to more general
conditions of the type%
\begin{align*}
\mathbf{G}_{3}u(s( t),t)   &  =g_{3}\bigl(t,s(t)  ,\dot{s}(  t)  \bigr)  ,\\
\mathbf{G}_{4}u(s(t),t)   &  =g_{4}\bigl(t,s(t),\dot{s}(t)\bigr)  ,
\end{align*}
for some first order linear differential operators
$\mathbf{G}_{3}$ and $\mathbf{G}_{4},$ and where $g_{3}$ and $g_{4}$ are
given functions of three variables.

\subsubsection{Nonlinear generalization}

For some FBPs we can only express the conditions \eqref{eq_FBP_b3} and
\eqref{eq_FBP_b4} as
\begin{align*}
g_{3}\bigl(  t,s(t)  ,\dot{s}(  t)  ,u(  s(t)  ,t)  ,u_{x}(  s(t),t)  \bigr)   &
=0,\\
g_{4}\bigl(  t,s(t)  ,\dot{s}(  t)  ,u(  s(t)  ,t)  ,u_{x}(  s(t),t)  \bigr)   &
=0,
\end{align*}
where $g_{3}$ and $g_{4}$ are given functions of five variables. In this case
the elegant decomposition into a linear and nonlinear problems will be lost, however the hierarchic structure of the minimization problem remains, i.e., the minimization Problem \ref{Problem_M} can be reduced to the Problem \ref{Problem_Mtil} with the only difference that the auxiliary problem \eqref{Linear LS} may be nonlinear.

\subsection{Nonlinear form of the boundary\label{SubSec_NL_FreeBoundary}}

Often in applications some additional information is available about the free
boundary structure. With this additional knowledge (or for other reasons) the
linear decomposition of the boundary might not be appropriate. It is easy to
see that we can still apply the algorithm assuming that
\[
s_{K}=\xi\left(  b_{0},...,b_{K}\right)  ,
\]
for the known function $\xi$. This nonlinear structure can slow down the
algorithm for the search of the minimum in Problems \ref{Problem_M} or
\ref{Problem_Mtil}, due to the constraints \eqref{eq:_probM_condS} and
\eqref{eq:_probMtil_condS} respectively. In our particular Matlab
implementation, we have used \texttt{fmincon} function that performs better with a linear constraint on the boundary than a nonlinear one.

\subsection{Concluding remarks}

A method for approximate solution of a large variety of FBPs is proposed. It
is based on a possibility to construct a complete system of solutions of
a parabolic equation called transmuted heat polynomials. The numerical
implementation is relatively simple and direct. The time required for
computations is within seconds. The method admits extensions onto a much
larger class of FBPs then that discussed in the paper.

\section*{Acknowledgements}
Research was supported by CONACYT, Mexico via the project
222478. The first named author would like to express his gratitude to the excellence scholarship granted by the Mexican Government via the Ministry of Foreign Affairs which gave him the opportunity to develop this work during his stay in the CINVESTAV, Mexico.

% Computational Mathematics and Mathematical Physics
% Computational and Applied Mathematics

\end{document}